\documentclass[11pt]{amsart}
\usepackage{amsrefs, graphicx}

\numberwithin{equation}{section}
\newtheorem{Theorem}{Theorem}[section]
\newtheorem{Lemma}[Theorem]{Lemma}

\theoremstyle{definition}
\newtheorem{Example}[Theorem]{Example}

\newcommand{\R}{\mathbb{R}}
\renewcommand{\epsilon}{\varepsilon}
\renewcommand{\phi}{\varphi}
\renewcommand{\[}{\begin{equation}}
\renewcommand{\]}{\end{equation}}

\title{Eigenvalue estimates on quantum graphs}
\author{Sinan Ariturk}
\address{Pontif\'icia Universidade Cat\'olica do Rio de Janeiro, Brazil}
\email{ariturk@mat.puc-rio.br}

\begin{document}

\maketitle

\begin{abstract}
On a finite connected metric graph, we establish upper bounds for the eigenvalues of the Laplacian.
These bounds depend on the length, the Betti number, and the number of pendant vertices.
For trees, these estimates are sharp.
We also establish sharp upper bounds for the spectral gap of the complete graph $K_4$.
The proofs are based on estimates for eigenvalues on graphs with Dirichlet conditions imposed at the pendant vertices.
\end{abstract}

\section{Introduction}

In this article, we study eigenvalues of finite quantum graphs.
A quantum graph consists of a graph, a metric, and a differential operator.
Let $G$ be a finite connected graph, possibly with loops and multiple edges.
Assume $G$ contains at least one edge, and let $E$ be the edge set of $G$.
A metric on $G$ is a function $\ell: E \to (0, \infty)$, assigning a finite positive length to each edge.
The pair $(G, \ell)$ is called a finite metric graph.
We identify each edge of $G$ with an interval in $\R$ of the same length.
We consider the eigenvalue problem for the Laplacian with Neumann vertex conditions, i.e.
\[
\label{laplaceneumann}
	\begin{cases}
		-f'' = \lambda f & \text{ over each edge} \\
		\sum_{e \sim v} f'(v) = 0 & \text{ at each vertex $v$} \\
	\end{cases}
\]
The sum in the Neumann condition at a vertex $v$ is taken over all edges $e$ which are incident to $v$, and the derivatives are taken in the direction away from $v$ into the edge $e$.
Note that a loop at $v$ contributes two terms to this sum, one in each direction.
Additionally, we require $f$ to be continuous over $G$.
These eigenvalues form a sequence which can be indexed so that
\[
	0 = \lambda_0(G, \ell) < \lambda_1(G, \ell) \le \lambda_2(G, \ell) \le \ldots
\]
The eigenvalue $\lambda_1$ is also called the spectral gap.

The dependence of the eigenvalues $\lambda_j(G,\ell)$ on the metric graph $(G, \ell)$ is complicated. 
We study bounds for the eigenvalues $\lambda_j(G, \ell)$ in terms of the graph $G$ and the length of the graph, defined by
\[
	L(G, \ell) = \sum_{e \in E} \ell(e)
\]
For a finite graph $G$, define $\Lambda_j(G)$ to be the smallest number such that for all metrics $\ell$ on $G$,
\[
\label{biglambda}
	\lambda_j(G,\ell) \le \frac{\Lambda_j(G)}{L(G,\ell)^2}
\]
Also define $\mu_j(G)$ to be the largest number such that for all metrics $\ell$ on $G$,
\[
\label{lilmu}
	\lambda_j(G,\ell) \ge \frac{\mu_j(G)}{L(G,\ell)^2}
\]

Nicaise~\cite{N}, Friedlander~\cite{F}, and Kurasov and Naboko~\cite{KN} showed that $\mu_1(G) \ge \pi^2$ for any $G$.
Moreover, Band and Levy \cite{BL} showed that if $G$ has a bridge, then $\mu_1(G)=\pi^2$.
They also showed that if $G$ is 2-edge-connected, then $\mu_1(G)=4 \pi^2$.
For a finite graph $G$, let $m(G)$ be the size of $G$, i.e. the number of edges.
Kennedy, Kurasov, Malenova, and Mugnolo~\cite{KKMM} showed that
\[
\label{kkmmlj}
	\Lambda_1(G) \le m(G)^2 \pi^2
\]
Moreover, equality hold in \eqref{kkmmlj} if and only if $G$ is a flower graph with one vertex and $m(G)$ loops or a dipole graph with two vertices and no loops.
For a finite graph $G$, let $p(G)$ be the number of pendant vertices.
Band and Levy~\cite{BL} showed that, if $m(G) \ge 3$, then 
\[
	\Lambda_1(G) \le \bigg( m(G) - \frac{p(G)}{2} \bigg)^2 \pi^2
\]
If $G$ is a finite tree, then Rohleder \cite{R} established improved bounds on $\Lambda_1(G)$.
He also gave upper bounds for higher eigenvalues.
Band and Levy~\cite{BL} improved these estimates by replacing the size of $G$ with the number of leaves $p(G)$.
If $G$ is a finite tree, they showed that
\[
\label{rbl}
	\Lambda_1(G) = \frac{p(G)^2 \pi^2}{4}
\]
Moreover, if $G$ is a finite tree, then for $j \ge 1$,
\[
\label{rblj}
	\Lambda_j(G) \le \frac{j^2 p(G)^2 \pi^2}{4}
\]

However, if $G$ is a finite tree with $p(G) \ge 3$ and $j \ge 2$, then equality is not attained in \eqref{rblj}.
In the following theorem, we determine $\Lambda_j(G)$ for $j \ge 2$ when $G$ is a finite tree.

\begin{Theorem}
\label{thmtree}
Let $G$ be a finite tree containing at least one edge.
For $j \ge 1$,
\[
\label{thmtree1}
	\Lambda_j(G) = \bigg( j -1 + \frac{p(G)}{2} \bigg)^2 \pi^2
\]
\end{Theorem}

Note that for the case $j=1$, we recover \eqref{rbl}.
For a finite connected graph which is not a tree, we establish upper bounds for $\Lambda_j(G)$ in terms of the number of pendant vertices and the Betti number.
Recall that if a finite connected graph $G$ has $m(G)$ edges and $n(G)$ vertices, then the Betti number $\beta(G)$ is given by
\[
	\beta(G) = m(G) - n(G) + 1
\]
Equivalently, the Betti number is the smallest number of edges which can be deleted from $G$ to obtain a tree.
In particular, a finite connected graph is a tree if and only if its Betti number is zero.

\begin{Theorem}
\label{thmneu}
Let $G$ be a finite connected graph containing at least one edge.
For $j \ge 1$,
\[
\label{thmneu1}
	\Lambda_j(G) \le \bigg( j -1 + 2\beta(G) + \frac{p(G)}{2} \bigg)^2 \pi^2
\]
\end{Theorem}

If $\beta(G) \ge 1$, then we expect that the bound \eqref{thmneu1} can be improved.
For certain graphs, we are able to obtain sharper estimates.
In particular, we determine the spectral gap of a graph $G$ which admits an induced tree of order $n(G)-1$, provided $G$ has no loops or pendant vertices.
Here $n(G)$ is the order of $G$, i.e. the number of vertices.

\begin{Theorem}
\label{thminforest}
Let $G$ be a finite connected graph containing at least one edge.
Assume $G$ has no loops and no pendant vertices.
Also assume $G$ admits an induced tree of order $n(G)-1$.
Then
\[
\label{thminforest1}
	\Lambda_1(G) = \Big(1+ \beta(G) \Big)^2 \pi^2
\]
\end{Theorem}

We also determine the spectral gap of the complete graph $K_4$.

\begin{Theorem}
\label{thmk4}
The spectral gap of the complete graph $K_4$ is given by
\[
\label{thmk41}
	\Lambda_1(K_4) = 16 \pi^2
\]
\end{Theorem}

The proofs of these results are based on bounds for eigenvalues on finite metric graphs with Dirichlet conditions imposed at some vertices.
Fix a finite connected metric graph $(G, \ell)$.
Let $D$ be a non-empty subset of the vertex set $V$, and let $N$ be the complement of $D$ in $V$.
The eigenvalue equation for the Laplacian with Dirichlet conditions imposed at the vertices in $D$ and Neumann conditions imposed at the vertices in $N$ is
\[
\label{laplacedirichlet}
	\begin{cases}
		-f'' = \lambda f & \text{ over each edge} \\
		f(v) = 0 & \text{ at each vertex $v$ in $D$} \\
		\sum_{e \sim v} f'(v) = 0 & \text{ at each vertex $v$ in $N$} \\
	\end{cases}
\]
We also require $f$ to be continuous over $G$.
These eigenvalues form a sequence which can be indexed so that
\[
	0 < \lambda_1(G, \ell; D) \le \lambda_2(G, \ell; D) \le \lambda_3(G, \ell; D) \le \ldots
\]
If $v_1, v_2,\ldots,v_d$ are the vertices in $D$, then we will also use the notation $\lambda_j(G, \ell; v_1, v_2, \ldots, v_d)$ in place of $\lambda_j(G, \ell; D)$.
Let $|D|$ denote the number of vertices in $D$.
The following lemma establishes bounds for eigenvalues with Dirichlet conditions imposed at the pendant vertices.

\begin{Lemma}
\label{lemdir}
Let $(G, \ell)$ be a finite connected metric graph containing at least one edge.
Let $D$ be a non-empty set of vertices in $G$ containing every pendant vertex.
For $j \ge 1$,
\[
\label{lemdir1}
	\lambda_j(G, \ell; D) \le \frac{(j-2+2\beta(G)+|D|)^2 \pi^2}{L(G, \ell)^2}
\]
If $G$ has no pendant vertices, then for $j \ge 1$,
\[
\label{lemdir2}
	\lambda_j(G, \ell) \le \frac{(j-1+2\beta(G))^2 \pi^2}{L(G, \ell)^2}
\]
\end{Lemma}

We use this lemma to prove Theorem \ref{thmtree}, Theorem \ref{thmneu}, and Theorem \ref{thminforest}.
To prove Theorem \ref{thmk4}, we also use the following lemma which establishes upper bounds for the first eigenvalue on the complete bipartite graph $K_{1,3}$ with Dirichlet conditions imposed at the leaves.
These bounds are sharper than \eqref{lemdir1} for many metrics $\ell$ on $K_{1,3}$.

\begin{Lemma}
\label{lemclaw}
Let $D$ be the set of leaves in the complete bipartite graph $K_{1,3}$.
Let $\ell$ be a metric on $K_{1,3}$.
Let $e_1, e_2, e_3$ be the edges of $K_{1,3}$, labelled so that
\[
	\ell(e_1) \ge \ell(e_2) \ge \ell(e_3)
\]
Then
\[
	\lambda_1(K_{1,3}, \ell; D) \le \frac{(10/9)^2 \pi^2}{(\ell(e_1)+\ell(e_3))^2}
\]
\end{Lemma}

In the second section of the article, we review background on quantum graphs.
This includes basic properties and lemmas, as well as examples of quantum graphs where the eigenvalues can be computed explicitly.
In the third section, we prove Lemma~\ref{lemdir} and use it to prove Theorem \ref{thmtree}, Theorem~\ref{thmneu}, and Theorem \ref{thminforest}.
In the fourth section, we prove Lemma \ref{lemclaw} and Theorem~\ref{thmk4}.

There are many related results concerning eigenvalues of quantum graphs.
The dependence of the eigenvalues on the lengths of edges was studied by Berkolaiko and Kuchment \cite{BK2} and Exner and Jex \cite{EJ}.
The behavior of the eigenvalues when an edge is deleted from a graph was considered by Kurasov, Malenova, and Naboko~\cite{KMN}.
An eigenvalue optimization result for graphs in $\R^n$ containining prescribed vertices was established by Buttazzo, Ruffini, and Velichkov \cite{BRV}.
Eigenvalues on regular trees were studied by Solomyak~\cite{S}.
Estimates for low eigenvalues were established by Demirel and Harrell \cite{DH} and Karreskog, Kurasov, and Trygg Kupersmidt~\cite{KKTK}.
The first eigenvalue of the $p$-Laplacian was considered by Del Pezzo and Rossi~\cite{DPR}.

\section{Background}

In this section, we briefly review background on quantum graphs.
For a thorough treatment, we refer to Berkolaiko and Kuchment~\cite{BK}.

The eigenvalues on finite metric graphs can be characterized variationally.
Let $(G, \ell)$ be a finite connected metric graph which contains at least one edge.
Identify each edge of $G$ with an interval in $\R$ of the same length.
Let $V$ be the vertex set of $G$, and let $D$ be a non-empty subset of $V$.
Let $H^1(G, \ell)$ denote the set of continuous functions $f:G \to \R$ which are in $H^1$ over each edge.
Let $H_0^1(G, \ell; D)$ be the subspace of $H^1(G, \ell)$ consisting of functions which vanish at the vertices in $D$.
Then for $j \ge 1$,
\[
\label{rq}
	\lambda_j(G, \ell; D) = \min_U \max_{f \in U} \frac{\int_G |f'|^2}{\int_G |f|^2 }
\]
The minimum is taken over all $j$-dimensional subspaces $U$ of $H_0^1(G, \ell; D)$.
Similarly, for $j \ge 0$,
\[
\label{krq}
	\lambda_j(G, \ell) = \min_W \max_{f \in W} \frac{\int_G |f'|^2}{\int_G |f|^2 }
\]
The minimum is taken over all $(j+1)$-dimensional subspaces $W$ of $H^1(G, \ell)$.

To describe the behavior of eigenvalues under scaling of the metric, let $c>0$ be a constant.
Then $c \cdot \ell$ is a metric on $G$.
For $j \ge 0$,
\[
\label{scaleneumann}
	\lambda_j(G, c \cdot \ell) = c^{-2} \lambda_j(G, \ell)
\]
Also, for $j \ge 1$,
\[
\label{scaledirichlet}
	\lambda_j(G, c \cdot \ell; D) = c^{-2} \lambda_j(G, \ell; D)
\]

Note that vertices of degree two with Neumann conditions do not play a significant role.
Let $A$ be a graph obtained from $G$ by subdiving an edge $e_0$ into two edges $e_1$ and $e_2$.
Let $\ell_A$ be a compatible metric on $A$, i.e. a metric such that $\ell_A(e_1)+\ell_A(e_2)=\ell(e_0)$ and $\ell_A(e) = \ell(e)$ for all edges $e \neq e_0$ in $G$.
Then for $j \ge 0$,
\[
\label{smoothneumann}
	\lambda_j(G, \ell) = \lambda_j(A, \ell_A)
\]
Also for $j \ge 1$,
\[
\label{smoothdirichlet}
	\lambda_j(G, \ell; D) = \lambda_j(A, \ell_A; D)
\]

\begin{Example}
Let $P$ be a path graph of order two, i.e. a connected graph with two vertices and one edge $e$.
If $\ell$ is a metric on $P$, then the eigenvalue problem is equivalent to an eigenvalue problem on the interval $[0,\ell(e)]$.
For $j \ge 0$,
\[
\label{pathneumann}
	\lambda_j(P, \ell) = \frac{j^2 \pi^2}{\ell(e)^2}
\]
If $v_1$ and $v_2$ are the vertices of $P$, then for $j \ge 1$,
\[
\label{pathdirichlet}
	\lambda_j(P, \ell; v_1,v_2) = \frac{j^2 \pi^2}{\ell(e)^2}
\]
Also, for $j \ge 1$,
\[
\label{pathmixed}
	\lambda_j(P, \ell; v_1 ) = \frac{(2j-1)^2 \pi^2}{4 \ell(e)^2}
\]
\end{Example}

\begin{Example}
\label{sharpstar}
A star graph $S$ is a tree containing a vertex $v$ such that every edge in $S$ is incident to $v$.
Let $p \ge 3$ and let $j \ge 1$.
Let $S$ be a star graph with $p$ edges.
Let $\ell$ be a metric on $S$ such that $p-1$ edges have length one, and the other edge has length $2j-1$.
Then $L(S, \ell)=2j+p-2$, and
\[
\label{sharpstareq}
	\lambda_j(S, \ell) = \frac{\pi^2}{4}
\]
To verify \eqref{sharpstareq}, note that the eigenvalues $\lambda_j(S, \ell; v)$ can be computed easily.
This is because the eigenvalue problem with Dirichlet conditions imposed at $v$ is equivalent to an eigenvalue problem on $n$ disjoint intervals.
In particular
\[
	\lambda_j(S, \ell; v) = \lambda_{j+1}(S, \ell; v) = \frac{\pi^2}{4}
\]
Moreover, eigenvalue interlacing \cite[Theorem 3.1.8]{BK} states that
\[
\label{evinterlace}
	\lambda_j(S, \ell; v) \le \lambda_j(S, \ell) \le \lambda_{j+1}(S, \ell;v)
\]
This establishes \eqref{sharpstareq}.
\end{Example}

\begin{Example}
\label{dipole}
Fix $m \ge 3$, and let $G$ be a dipole graph with $m$ edges.
That is, let $G$ be a graph with two vertices $v$ and $w$, and $m$ edges, each of which is incident to both $v$ and $w$.
Let $\ell$ be the metric on $G$ such that $\ell(e)=1$ for every edge $e$ in $G$.
Then $L(G,\ell)=m$, and
\[
\label{dipole1}
	\lambda_1(G, \ell) = \pi^2
\]
Note that it is easy to show that $\pi^2$ is an eigenvalue.
To verify \eqref{dipole1}, it suffices to show that there is no eigenvalue $\lambda$ satisfying $0 < \lambda< \pi^2$.
The metric graph $(G, \ell)$ admits an isometry which maps $v$ to $w$ and maps each edge to itself.
It follows that every eigenvalue admits an eigenfunction which is either even or odd with respect to this isometry.
Using this observation, it is easy to show that there is no eigenvalue $\lambda$ satisfying $0 < \lambda < \pi^2$.
\end{Example}

The following lemma shows that contracting an edge yields a graph with larger eigenvalues.

\begin{Lemma}
\label{contract}
Let $(G, \ell)$ be a finite connected metric graph which contains at least two edges.
Let $e$ be an edge in $G$, and let $A$ be the graph obtained from $G$ by contracting $e$.
Let $\ell_A$ be the induced metric on $A$.
Then for $j \ge 1$,
\[
\label{contract1}
	\lambda_j(G, \ell) \le \lambda_j(A, \ell_A)
\]
Let $D$ be a non-empty set of vertices in $G$, and let $D_A$ be the induced set in $A$.
Then for $j \ge 1$,
\[
\label{contract2}
	\lambda_j(G, \ell; D) \le \lambda_j(A, \ell_A; D_A)
\]
\end{Lemma}

\begin{proof}
Let $f$ be a function in $H^1(A, \ell_A)$.
There is a unique function $g$ in $H^1(G, \ell)$ which is constant on $e$ and agrees with $f$ over $G \setminus e$.
Moreover,
\[
	\frac{\int_G |g'|^2}{\int_G |g|^2 } \le \frac{\int_A |f'|^2}{\int_A |f|^2 }
\]
Furthermore, if $f$ is in $H_0^1(A, \ell_A; D_A)$, then $g$ is in $H_0^1(G, \ell; D)$.
Because of \eqref{rq} and \eqref{krq}, this implies \eqref{contract1} and \eqref{contract2}.
\end{proof}

A consequence of Lemma \ref{contract} is that shortening an edge yields a metric with larger eigenvalues.

\begin{Lemma}
\label{shorten}
Let $(G, \ell_1)$ be a finite connected metric graph which contains at least one edge.
Let $e_0$ be an edge in $G$, and let $\ell_2$ be a metric on $G$ such that $\ell_2(e_0) < \ell_1(e_0)$ and $\ell_2(e)= \ell_1(e)$ for all edges $e \neq e_0$.
Then for $j \ge 1$,
\[
\label{shorten1}
	\lambda_j(G, \ell_1) \le \lambda_j(G, \ell_2)
\]
Let $D$ be a non-empty set of vertices in $G$.
Then for $j \ge 1$,
\[
\label{shorten2}
	\lambda_j(G, \ell_1; D) \le \lambda_j(G, \ell_2; D)
\]
\end{Lemma}

\begin{proof}
Note that $(G, \ell_2)$ can be obtained from $(G, \ell_1)$ by subdividing $e_0$ into two edges and then contracting one of the new edges.
Therefore \eqref{shorten1} and \eqref{shorten2} follow from \eqref{smoothneumann}, \eqref{smoothdirichlet}, and Lemma \ref{contract}.
\end{proof}

Lemma \ref{contract} also shows that deleting a pendant edge and the incident pendant vertex from a graph increases the Neumann eigenvalues.

\begin{Lemma}
\label{neumanndelete}
Let $(G, \ell)$ be a finite connected metric graph which contains at least two edges.
Assume $G$ contains a pendant edge $e$.
Let $A$ be the finite graph obtained from $G$ by deleting $e$ and the incident pendant vertex.
Let $\ell_A$ be the induced metric.
For any $j$,
\[
\label{neumanndeleteeq}
	\lambda_j(G, \ell) \le \lambda_j(A, \ell_A)
\]
\end{Lemma}

\begin{proof}
This follows immediately from Lemma \ref{contract}.
\end{proof}

The following lemma describes the effect of deleting a pendant edge when Dirichlet conditions are imposed at the pendant vertices.

\begin{Lemma}
\label{dirichletdelete}
Let $(G, \ell)$ be a finite connected metric graph which contains at least two edges.
Let $D$ be a set of vertices in $G$ which contains every pendant vertex.
Let $v$ be a pendant vertex and let $e$ be the incident pendant edge.
View $e$ as a path subgraph and let $\ell_e$ be the induced metric.
Let $A$ be the finite graph obtained from $G$ by deleting $e$ and $v$.
Let $\ell_A$ be the induced metric, and let $D_A=D \setminus \{ v \}$.
Assume $D_A$ is not empty.
For any $j \ge 1$,
\[
\label{dirichletdeleteeq}
	\lambda_j(G, \ell; D) \le \max \bigg( \lambda_j(A, \ell_A; D_A), \lambda_1(e, \ell_e; v) \bigg)
\]
\end{Lemma}

\begin{proof}
Let $w$ be the vertex incident to $e$ which is not $v$.
Identify $e$ with the interval $[0,\ell(e)]$ so that $v$ is identified with zero and $w$ is identified with $\ell(e)$.
For a function $f$ in $H_0^1(A, \ell_A; D_A)$, define $g$ in $H_0^1(G, \ell; D)$ by
\[
	g(x) =
	\begin{cases}
		f(x) & x \in A \\
		f(w) \sin \Big( \frac{x \pi}{2\ell(e)} \Big) & x \in e \sim [0,\ell(e)] \\
	\end{cases}
\]
Then
\[
	\frac{\int_G |g'|^2}{\int_G |g|^2 } \le \max \bigg( \frac{\int_A |f'|^2}{\int_A |f|^2 }, \lambda_1(e, \ell_e; v) \bigg)
\]
Because of \eqref{rq}, this implies \eqref{dirichletdeleteeq}.
\end{proof}

We conclude this section by establishing a basic fact about trees which will be used in the proof of Lemma \ref{lemdir}.

\begin{Lemma}
\label{degreetwo}
Let $G$ be a finite tree which contains at least two edges.
Assume that there is no pair of incident leaf edges in $G$. 
Then there is a vertex of degree two in $G$ which is adjacent to a leaf.
\end{Lemma}

\begin{proof}
Let $A$ be the tree obtained from $G$ by deleting each of the leaves and leaf edges.
Note that $A$ contains at least one edge.
Let $w$ be a leaf of $A$.
There is no pair of incident leaf edges in $G$, so the degree of $w$ in $G$ is at most two.
The leaves of $G$ are not in $A$, so the degree of $w$ in $G$ is exactly two.
In particular, $w$ is adjacent to a leaf of $G$.
\end{proof}

\section{Eigenvalue estimates}

In this section, we prove Lemma \ref{lemdir} and use it to prove Theorem \ref{thmtree}, Theorem \ref{thmneu}, and Theorem \ref{thminforest}.
We first prove Lemma \ref{lemdir} for trees.

\begin{Lemma}
\label{lemdtree}
Let $(G, \ell)$ be a finite metric tree containing at least one edge.
Let $D$ be a set of vertices in $G$ which contains every leaf.
For $j \ge 1$,
\[
\label{lemdtreeeq}
	\lambda_j(G, \ell; D) \le \frac{(j+|D|-2)^2 \pi^2}{L(G, \ell)^2}
\]
\end{Lemma}

\begin{proof}
Note that if $|D|=2$, then \eqref{lemdtreeeq} follows from \eqref{smoothdirichlet} and \eqref{pathdirichlet}.
We complete the proof by induction on $j+|D|$.
Fix a finite metric tree $(G, \ell)$ which contains at least one edge.
Also fix a set of vertices $D$ in $G$ which contains every leaf and fix a positive integer $j$.
We may assume that $|D| \ge 3$.
Let $(A, \ell_A)$ be a finite metric tree, let $D_A$ be a set of vertices in $A$ containing every leaf, and let $i$ be a positive integer.
By induction, we may assume that if $|D_A| + i < |D| + j$, then
\[
\label{ldin}
	\lambda_i(A, \ell_A; D_A) \le \frac{(i+|D_A|-2)^2 \pi^2}{L(A, \ell_A)^2}
\]

By \eqref{smoothdirichlet}, we may assume that there are no vertices of degree two in $G$ which are not in $D$.
By \eqref{scaledirichlet}, we may assume that $L(G, \ell)=j+|D|-2$.
To prove \eqref{lemdtreeeq}, we need to show that
\[
\label{lemdtree1}
	\lambda_j(G, \ell; D) \le \pi^2
\]
We break the argument into four cases.
In the first case, we assume that there is a leaf edge $e$ such that $\ell(e) > 1$.
In the second case, we assume that there is a leaf edge $e$ such that $1/2 \le \ell(e) \le 1$. 
In the third case, we assume that $\ell(e) \le 1/2$ for every leaf edge $e$ and there is a pair of incident leaf edges.
In the fourth case, we assume that $\ell(e) \le 1/2$ for every leaf edge $e$ and there is no pair of incident leaf edges.
In each case we use the following notation.
Let $p=p(G)$ be the number of leaves in $G$ and let $v_1,v_2, \ldots, v_p$ be the leaves.
For each $i=1,2,\ldots,p$, let $e_i$ be the leaf edge of $G$ which is incident to $v_i$.

\emph{Case 1:}
In this case, we assume that there is a leaf edge $e$ in $G$ such that $\ell(e) > 1$.
Without loss of generality, we may assume that $\ell(e_1) > 1$.
Subdivide $e_1$ into two edges.
Define a compatible metric so that the new edge incident to $v_1$ has length one and the other new edge has length $\ell(e_1)-1$.
Let $w$ be the new vertex adjacent to $v_1$.
Let $B$ be the new edge incident to $v_1$ and $w$.
View $B$ as a path subgraph and let $\ell_B$ be the induced metric.
Then
\[
\label{dnj11}
	\lambda_1(B, \ell_B; v_1, w) = \pi^2
\]
If $j=1$, then this establishes \eqref{lemdtree1}, because \eqref{rq} implies that
\[
	\lambda_1(G, \ell; D) \le \lambda_1(B, \ell_B; v_1,w) = \pi^2
\]
Therefore, we may assume that $j \ge 2$.
Let $A$ be the tree obtained from $G$ by deleting $B$ and $v_1$.
Let $\ell_A$ be the induced metric.
Define
\[
	D_A = D \cup \{ w \} \setminus \{ v_1 \}
\]
Note $D_A$ contains every leaf of $A$.
Also $|D_A|=|D|$ and $L(A, \ell_A) = j+|D|-3$, so by \eqref{ldin},
\[
\label{dnj12}
	\lambda_{j-1}(A, \ell_A; D_A) \le \pi^2
\]
Because of \eqref{rq}, the bounds \eqref{dnj11} and \eqref{dnj12} establish \eqref{lemdtree1}.

\emph{Case 2:}
In this case, we assume that there is a leaf edge $e$ in $G$ such that $1/2 \le \ell(e) \le 1$.
Without loss of generality, we may assume that $1/2 \le \ell(e_1) \le 1$.
View $e_1$ as a path subgraph and let $\ell_1$ be the induced metric.
Then
\[
\label{dnj21}
	\lambda_1(e_1, \ell_1; v_1) = \frac{\pi^2}{4\ell(e_1)^2} \le \pi^2
\]
Let $A$ be the tree obtained from $G$ by deleting $e_1$ and $v_1$.
Let $\ell_A$ be the induced metric and define $D_A = D \setminus \{ v_1 \}$.
Recall that, by assumption, there are no vertices in $G$ of degree two which are not in $D$.
Therefore, $D_A$ contains every leaf of $A$.
Also $|D_A| = |D|-1$ and $L(A, \ell_A) \ge j+|D|-3$, so by \eqref{ldin},
\[
\label{dnj22}
	\lambda_j(A, \ell_A; D_A) \le \pi^2
\]
By Lemma \ref{dirichletdelete}, the bounds \eqref{dnj21} and \eqref{dnj22} establish \eqref{lemdtree1}.

\emph{Case 3:}
In this case, we assume that $\ell(e) \le 1/2$ for every leaf edge $e$ and there is a pair of incident leaf edges.
Without loss of generality, we may assume that $e_1$ and $e_2$ are incident.
Let $w$ be the vertex incident to $e_1$ and $e_2$.
Let $A$ be the tree obtained from $G$ by deleting $e_1$, $e_2$, $v_1$, and $v_2$.
Let $\ell_A$ be the induced metric, and define
\[
	D_A = D \cup \{ w \} \setminus \{ v_1, v_2 \}
\]
Note that $D_A$ contains every leaf of $A$.
Moreover $|D_A|=|D|-1$ and $L(A, \ell_A) \ge j+|D|-3$.
By \eqref{ldin},
\[
	\lambda_j(A, \ell_A; D_A) \le \pi^2
\]
Then \eqref{rq} implies that
\[
	\lambda_j(G, \ell; D) \le \lambda_j(A, \ell_A; D_A) \le \pi^2
\]
This establishes \eqref{lemdtree1}.

\emph{Case 4:}
In this case, we assume that $\ell(e) \le 1/2$ for every leaf edge $e$ and there is no pair of incident leaf edges.
By Lemma \ref{degreetwo}, there is a vertex $w$ of degree two which is adjacent to a leaf.
Without loss of generality, we may assume that $w$ is adjacent to $v_1$.
By assumption, there are no vertices in $G$ of degree two which are not in $D$, so $w$ is in $D$.
Let $A$ be the graph obtained from $G$ by deleting $e_1$ and $v_1$.
Let $\ell_A$ be the induced metric and let $D_A=D \setminus \{ v_1 \}$.
Then $D_A$ contains every pendant vertex of $A$.
Also $|D_A|=|D|-1$ and $L(A, \ell_A) \ge j + |D| - 3$.
By \eqref{ldin},
\[
	\lambda_j(A, \ell_A; D_A) \le \pi^2
\]
Then \eqref{rq} implies that
\[
	\lambda_j(G, \ell; D) \le \lambda_j(A, \ell_A; D_A) \le \pi^2
\]
This establishes \eqref{lemdtree1}.
\end{proof}

Now we use Lemma \ref{lemdtree} to prove Lemma \ref{lemdir}.

\begin{proof}[Proof of Lemma 1.5]
If $\beta(G)=0$, then \eqref{lemdir1} follows from Lemma \ref{lemdtree}, while \eqref{lemdir2} is vacuous.
We complete the proof by induction on $\beta(G)$.
Fix a finite connected metric graph $(G, \ell)$ with at least one edge.
Fix a non-empty set of vertices $D$ which contains every pendant vertex.
We may assume that $\beta(G) \ge 1$.
Let $(A, \ell_A)$ be a finite connected metric graph with at least one edge.
Let $D_A$ be a non-empty set of vertices in $A$ which contains every pendant vertex.
By induction, we may assume that if $\beta(A) < \beta(G)$, then for $j \ge 1$,
\[
\label{ound24}
	\lambda_j(A, \ell_A; D_A) \le \frac{(j-2+2\beta(A)+|D_A|)^2 \pi^2}{L(A, \ell_A)^2}
\]

Since $\beta(G) \ge 1$, there is an edge in $G$ which can be deleted to obtain a finite connected graph with Betti number $\beta(G)-1$.
Form a graph $B$ from $G$ by subdividing this edge.
Let $v$ be the new vertex.
Let $\ell_B$ be a compatible metric on $B$.
Let $A$ be the graph obtained from $B$ by splitting $v$ into two pendant vertices $v_1$ and $v_2$.
Let $\ell_A$ be the induced metric.
Then $A$ is a finite connected graph and $\beta(A)=\beta(G)-1$.
Also $L(A, \ell_A) = L(G, \ell)$.
Define
\[
	D_A = D \cup \{ v_1, v_2 \}
\]
Then $D_A$ contains every pendant vertex of $A$, and $|D_A|=|D|+2$.
By \eqref{ound24}, for $j \ge 1$,
\[
\begin{split}
	\lambda_j(G, \ell; D) &\le \lambda_j(G, \ell; D \cup \{ v \} ) \\
		&= \lambda_j(A, \ell_A; D_A) \\
		&\le \frac{(j-2+2\beta(G)+|D|)^2 \pi^2}{L(G, \ell)^2} \\
\end{split}
\]
This establishes \eqref{lemdir1}.
Similarly, by \eqref{ound24}, for $j \ge 1$,
\[
\begin{split}
	\lambda_j(G, \ell) &\le \lambda_{j+1}(G, \ell; v ) \\
		&= \lambda_{j+1}(A, \ell_A; v_1, v_2 ) \\
		&\le \frac{(j-1+2\beta(G))^2 \pi^2}{L(G, \ell)^2} \\
\end{split}
\]
This establishes \eqref{lemdir2}.
\end{proof}

Now we can prove Theorem \ref{thmneu}.

\begin{proof}[Proof of Theorem 1.2]
Note that the case $p(G)=0$ follows from \eqref{lemdir2}.
We complete the proof of \eqref{thmneu1} by induction on $p(G)$.
Let $G$ be a finite connected graph.
We may assume $p(G) \ge 1$.
By induction, we may assume that if $A$ is a finite connected graph with $p(A) < p(G)$, then for $j \ge 1$,
\[
\label{bnin}
	\Lambda_j(A) \le \bigg( j-1 + 2\beta(A) + \frac{p(A)}{2} \bigg)^2 \pi^2
\]
Let $\ell$ be a metric on $G$.
By \eqref{smoothneumann}, we may assume there are no vertices of degree two in $G$.
By \eqref{scaleneumann}, we may assume $L(G, \ell)=2j-2+4\beta(G)+p(G)$.
It suffices to prove
\[
\label{thmneumanneq2}
	\lambda_j(G, \ell) \le \frac{\pi^2}{4}
\]
The case $\beta(G)=0$ and $p(G)=2$ follows from \eqref{smoothneumann} and \eqref{pathneumann}, so if $\beta(G)=0$, then we may assume that $p(G) \ge 3$.

We break the argument into two cases.
In the first case, we assume that $\ell(e) > 1$ for every pendant edge $e$.
In the second case, we assume that there is a pendant edge $e$ such that $\ell(e) \le 1$.
In both cases, we use the following notation.
Let $p=p(G)$ be the number of pendant vertices in $G$ and let $v_1,v_2, \ldots, v_p$ be the pendant vertices.
For each $i=1,2,\ldots,p$, let $e_i$ be the pendant edge of $G$ which is incident to $v_i$.

\emph{Case 1:}
In this case, we assume that $\ell(e) > 1$ for every pendant edge $e$.
Subdivide every pendant edge, and define a compatible metric so that the new pendant edges each have length one.
For each $i=1,2,\ldots,p$, let $w_i$ be the new vertex adjacent to $v_i$, and let $B_i$ be the new edge incident to $v_i$ and $w_i$.
View $B_i$ as a path subgraph and let $\ell_i$ be the induced metric.
For each $i=1,2,\ldots,p$,
\[
\label{uivi1}
	\lambda_1(B_i, \ell_i; w_i) = \frac{\pi^2}{4}
\]
Because of \eqref{rq} and \eqref{krq}, this establishes that
\[
	\lambda_{p-1}(G, \ell) \le \frac{\pi^2}{4}
\]
If $j \le p-1$, then this establishes \eqref{thmneumanneq2}.
Therefore, we may assume $j \ge p$.
Let $A$ be the graph obtained by deleting $B_1, B_2, \ldots, B_p$ and $v_1, v_2, \ldots, v_p$.
Let $\ell_A$ be the induced metric.
Note that $p(A)=p(G)$ and $\beta(A)=\beta(G)$.
Also,
\[
	L(A, \ell_A)=2j-2+4\beta(G)
\]
By Lemma \ref{lemdir},
\[
\label{jn1s}
	\lambda_{j-p+1}(A, \ell_A; w_1, w_2, \ldots, w_p) \le \frac{\pi^2}{4}
\]
Because of \eqref{rq} and \eqref{krq}, the bounds \eqref{uivi1} and \eqref{jn1s} establish \eqref{thmneumanneq2}.

\emph{Case 2:}
In this case, we assume that $\ell(e) \le 1$ for some pendant edge $e$.
Without loss of generality, we may assume that $\ell(e_1) \le 1$.
Let $A$ be the graph obtained from $G$ by deleting $e_1$ and $v_1$.
Let $\ell_A$ be the induced metric.
By assumption, there are no vertices of degree two in $G$, so $p(A)=p(G)-1$.
Also $\beta(A)=\beta(G)$.
Note that
\[
	L(A, \ell_A) \ge 2j-2+4 \beta(A) + p(A)
\]
Therefore, by Lemma \ref{neumanndelete} and \eqref{bnin},
\[
	\lambda_j(G, \ell) \le \lambda_j(A, \ell_A) \le \frac{\pi^2}{4}
\]
This establishes \eqref{thmneumanneq2}.
\end{proof}

Next we prove Theorem \ref{thmtree}.

\begin{proof}[Proof of Theorem 1.1]
By Theorem \ref{thmneu},
\[
\label{treepf1}
	\Lambda_j(G) \le \bigg( j-1+ \frac{p(G)}{2} \bigg)^2 \pi^2
\]
Let $S$ be a star graph with $p(G)$ edges.
By Example \ref{sharpstar}, we have
\[
\label{treepf2}
	\Lambda_j(S) \ge \bigg( j-1+\frac{p(G)}{2} \bigg)^2 \pi^2
\]
Note that $S$ can be obtained from $G$ by contracting edges.
By a continuity result established by Band and Levy \cite[Appendix A]{BL}, this implies that $\Lambda_j(G) \ge \Lambda_j(S)$.
Therefore \eqref{treepf1} and \eqref{treepf2} establish \eqref{thmtree1}.
\end{proof}

We conclude this section by proving Theorem \ref{thminforest}.

\begin{proof}[Proof of Theorem 1.3]
We first prove that
\[
\label{inforestpf1}
	\Lambda_1(G) \le \Big( 1 + \beta(G) \Big)^2 \pi^2
\]
Let $\ell$ be a metric on $G$.
By \eqref{scaleneumann}, we may assume that
\[
	L(G, \ell) = 1 + \beta(G)
\]
It suffices to prove that
\[
\label{inforestpf2}
	\lambda_1(G, \ell) \le \pi^2
\]
Let $n=n(G)$ be the order of $G$, and let $v_1, v_2, v_3, \ldots, v_n$ be the vertices of $G$.
By assumption there are $n-1$ vertices which induce a tree in $G$.
Without loss of generality, we may assume the vertices $v_2, v_3, \ldots, v_n$ induce a tree in $G$.
Note that the degree of $v_1$ is $\beta(G)+1$.
Let $A$ be the tree obtained from $G$ by splitting $v_1$ into $\beta(G)+1$ leaves.
Let $\ell_A$ be the induced metric on $A$.
Let $D$ be the set of leaves in $A$.
Then by \eqref{rq} and \eqref{krq},
\[
	\lambda_1(G, \ell) \le \lambda_2(G, \ell; v_1) = \lambda_2(A, \ell_A; D)
\]
Note that $|D|=\beta(G)+1$.
Also $L(A, \ell_A) = L(G, \ell)$.
Hence, by Lemma \ref{lemdtree},
\[
	\lambda_2(A, \ell_A; D) \le \frac{|D|^2 \pi^2}{L(A, \ell_A)^2} = \pi^2
\]
This proves \eqref{inforestpf2}, establishing \eqref{inforestpf1}.

Let $B$ be a dipole graph with $\beta(G)+1$ edges.
By Example \ref{dipole},
\[
\label{inforestpf3}
	\Lambda_1(B) \ge \Big( 1 + \beta(G) \Big)^2 \pi^2
\]
Note that $B$ can be obtained from $G$ by contracting edges.
By a continuity result established by Band and Levy \cite[Appendix A]{BL}, this implies that $\Lambda_1(G) \ge \Lambda_1(B)$.
Therefore \eqref{inforestpf1} and \eqref{inforestpf3} establish \eqref{thminforest1}.
\end{proof}

\section{The Spectral Gap of $K_4$}

In this section we prove Lemma \ref{lemclaw} and Theorem \ref{thmk4}.
We first prove Lemma \ref{lemclaw}.

\begin{proof}[Proof of Lemma 1.6]
By \eqref{scaledirichlet}, we may assume that $\ell(e_1)+\ell(e_3)=10/9$.
By Lemma \ref{shorten}, we may assume that $\ell(e_2)=\ell(e_3)$.
It suffices to prove that
\[
\label{clawpf}
	\lambda_1(K_{1,3}, \ell; D) \le \pi^2
\]
For each $i=1,2,3$, let $v_i$ be the leaf of $K_{1,3}$ incident to $e_i$.
Let $w$ be the vertex of degree three in $K_{1,3}$.
If $\ell(e_1) \ge 1$, then view $e_1$ as a path subgraph and let $\ell_1$ be the induced metric.
Then \eqref{clawpf} follows, because
\[
	\lambda_1(K_{1,3}, \ell; D) \le \lambda_1(e_1, \ell_1; v_1, w) \le \pi^2
\] 
Therefore, we may assume that $\ell(e_1) < 1$.
Identify the edge $e_1$ with the interval $[0,\ell(e_1)]$ so that $v_1$ is identified with zero and $w$ is identified with $\ell(e_1)$.
Similarly, identify each edge $e_2$ and $e_3$ with the interval $[0, \ell(e_3)]$ so that $v_2$ and $v_3$ are each identified with zero and $w$ is identified with $\ell(e_3)$.
Define $f$ in $H_0^1(K_{1,3}, \ell; D)$ by
\[
	f(t) =
	\begin{cases}
		\sin (\pi \ell(e_3) ) \sin (\pi t) & t \text{ in } e_1 \sim [0, \ell(e_1)] \\
		\sin (\pi \ell(e_1) ) \sin (\pi t) & t \text{ in } e_2 \sim [0, \ell(e_3)] \\
		\sin (\pi \ell(e_1) ) \sin (\pi t) & t \text{ in } e_3 \sim [0, \ell(e_3)] \\
	\end{cases}
\]
Note that $f \ge 0$ because $\ell(e_1) < 1$.
We claim that
\[
\label{clawclaim}
	\sin (\pi \ell(e_3)) \cos (\pi \ell(e_1)) + 2 \sin (\pi \ell(e_1)) \cos (\pi \ell(e_3)) < 0
\]
Assuming this claim, integration by parts shows that
\[
	\frac{\int_{K_{1,3}} |f'|^2}{\int_{K_{1,3}} |f|^2} \le \pi^2
\]
Therefore \eqref{clawpf} follows.
It remains to prove \eqref{clawclaim}.
We have the identity
\[
\label{clawclaimpf1}
	\sin x \cos y + 2 \sin y \cos x = (1/2) \Big( 3 \sin(x+y) - \sin(x-y) \Big)
\]
Furthermore,
\[
\label{clawclaimpf2}
	3 \sin \Big( \pi(\ell(e_1)+\ell(e_3)) \Big) = 3 \sin \bigg( \frac{10 \pi}{9} \bigg) < -1 \le \sin( \pi(\ell(e_1)-\ell(e_3)))
\]
Now \eqref{clawclaimpf1} and \eqref{clawclaimpf2} establish the claim \eqref{clawclaim}, completing the proof.
\end{proof}

We prove Theorem \ref{thmk4} by breaking the argument into several cases.
Each of the following lemmas treat various special cases.
For the rest of the section, we use the following notation.
Denote the vertices of $K_4$ by $w$, $x$, $y$, and $z$.
Let $wx$ be the edge incident to $w$ and $x$.
Similarly, denote the other edges by $wy$, $wz$, $xy$, $xz$, and $yz$.

\begin{Lemma}
\label{k4long}
Let $\ell$ be a metric on $K_4$ such that $L(K_4, \ell)=4$.
Assume there is an edge in $K_4$ of length greater than or equal to one.
Then
\[
\label{k4long1}
	\lambda_1(K_4, \ell) \le \pi^2
\]
\end{Lemma}

\begin{proof}
Without loss of generality, we may assume that $\ell(xz) \ge 1$.
We break the argument into two cases.
In the first case we assume $\ell(xz)+\ell(wy) \ge 2$.
In the second case we assume $\ell(xz)+\ell(wy) \le 2$.

\emph{Case 1:}
In this case, we assume $\ell(xz)+\ell(wy) \ge 2$.
Let $G$ be the graph obtained from $K_4$ by contracting the other four edges $xy$, $yz$, $wx$, and $wz$.
Let $\ell_G$ be the induced metric on $G$.
By Lemma \ref{contract},
\[
	\lambda_1(K_4, \ell) \le \lambda_1(G, \ell_G)
\]
Moreover $G$ has two edges and $L(G, \ell_G) \ge 2$.
Therefore \eqref{k4long1} follows from the bound \eqref{kkmmlj}.

\emph{Case 2:}
In this case, we assume $\ell(xz) + \ell(wy) \le 2$.
Let $G$ be the graph obtained from $K_4$ by deleting the edges $xz$ and $wy$.
Then $G$ is a cycle graph and $L(G, \ell_G) \ge 2$, so
\[
	\lambda_1(G, \ell_G) \le \pi^2
\]
Moreover, there is a corresponding eigenfunction $\phi$ in $H^1(G, \ell_G)$ such that $\phi(w)=\phi(y)$.
Let $A$ be the graph obtained from $K_4$ by subdividing $xz$.
Let $v$ be the new vertex and let $vx$ and $vz$ be the new edges incident to $x$ and $z$, respectively.
Let $\ell_A$ be the compatible metric on $A$ such that
\[
	\ell_A(vx)=\ell_A(vz)=\frac{\ell(xz)}{2}
\]
Extend $\phi$ to $wy$ so that $\phi$ is constant over $wy$.
Identify $vx$ with the interval $[0, \ell(xz)/2]$ so that $v$ is identified with zero and $x$ is identified with $\ell(xz)/2$.
Extend $\phi$ to $vx$ so that if $t$ is in $vx \sim [0 ,\ell(xz)/2]$, then
\[
	\phi(t) = \phi(x) \sin \bigg( \frac{\pi t}{\ell(xz)} \bigg)
\]
Also identify $vz$ with the interval $[0, \ell(xz)/2]$ so that $v$ is identified with zero and $z$ is identified with $\ell(xz)/2$.
Extend $\phi$ to $vz$ so that if $t$ is in $vz \sim [0 ,\ell(xz)/2]$, then
\[
	\phi(t) = \phi(z) \sin \bigg( \frac{\pi t}{\ell(xz)} \bigg)
\]
Note that $\phi$ attains positive and negative values, and let $\phi^+$ and $\phi^-$ be the positive and negative parts of $\phi$, respectively.
That is, define $\phi^+= \max(\phi,0)$ and $\phi^-=-\min(\phi,0)$.
Then $\phi^+$ and $\phi^-$ are in $H^1(K_4, \ell)$.
Let $V$ be the subspace of $H^1(K_4, \ell)$ generated by $\phi^+$ and $\phi^-$.
Then $V$ is two-dimensional, and
\[
	\lambda_1(K_4, \ell) \le \max_{f \in V} \frac{ \int_{K_4} |f'|^2 }{ \int_{K_4} |f|^2 } \le \pi^2
\]
This establishes \eqref{k4long1}.
\end{proof}

\begin{Lemma}
\label{k4ef}
Let $\ell$ be a metric on $K_4$ such that $L(K_4, \ell)=4$.
Assume $\ell(xy)+\ell(xz)+\ell(yz) \ge 2$.
Also assume $\ell(wy) \ge 1/2$ and $\ell(wz) \ge 1/2$.
Then
\[
\label{k4ef1}
	\lambda_1(K_4, \ell) \le \pi^2
\]
\end{Lemma}

\begin{proof}
Let $G$ be the graph obtained from $K_4$ by deleting $wx$, $wy$, and $wz$.
Let $\ell_G$ be the induced metric.
Then $G$ is a cycle graph and $L(G, \ell_G) \ge 2$, so
\[
	\lambda_1(G, \ell_G) \le \pi^2
\]
Moreover, there is a corresponding eigenfunction $\phi$ in $H^1(G, \ell_G)$ such that $\phi(x)=0$.
Extend $\phi$ to $wx$ so that $\phi$ is identically zero over $wx$.
Identify $wy$ with the interval $[0, \ell(wy)]$ so that $w$ is identified with zero and $y$ is identified with $\ell(wy)$.
Extend $\phi$ to $wy$ so that if $t$ is in $wy \sim [0 ,\ell(wy)]$, then
\[
	\phi(t) = \phi(y) \sin \bigg( \frac{\pi t}{2 \ell(wy)} \bigg)
\]
Identify $wz$ with the interval $[0, \ell(wz)]$ so that $w$ is identified with zero and $z$ is identified with $\ell(wz)$.
Extend $\phi$ to $wz$ so that if $t$ is in $wz \sim [0 ,\ell(wz)]$, then
\[
	\phi(t) = \phi(z) \sin \bigg( \frac{\pi t}{2 \ell(wz)} \bigg)
\]
Note that $\phi$ attains positive and negative values, and let $\phi^+$ and $\phi^-$ be the positive and negative parts of $\phi$, respectively.
Then $\phi^+$ and $\phi^-$ are in $H^1(K_4, \ell)$.
Let $V$ be the subspace of $H^1(K_4, \ell)$ generated by $\phi^+$ and $\phi^-$.
Then
\[
	\lambda_1(K_4, \ell) \le \max_{f \in V} \frac{ \int_{K_4} |f'|^2 }{ \int_{K_4} |f|^2 } \le \pi^2
\]
This proves \eqref{k4ef1}.
\end{proof}

\begin{Lemma}
\label{k4def}
Let $\ell$ be a metric on $K_4$ such that $L(K_4, \ell)=4$.
Assume
\[
\label{k4def1}
	\max \bigg( \ell(wx), \ell(wy), \ell(wz) \bigg) < 1/2
\]
Then
\[
	\lambda_1(K_4, \ell) \le \pi^2
\]
\end{Lemma}

\begin{proof}
Without loss of generality, we may assume that
\[
\label{k4def2}
	\ell(wz) \le \min( \ell(wx), \ell(wy) )
\]
and
\[
\label{k4def3}
	\ell(yz) + \ell(wy) \le \ell(xz) + \ell(wx)
\]
By Lemma \ref{k4long}, we may assume that
\[
\label{k4def4}
	\max \bigg( \ell(xy), \ell(xz), \ell(yz) \bigg) < 1
\]
By \eqref{k4def1} and \eqref{k4def4},
\[
\label{k4def5}
	\min \bigg( \ell(xy), \ell(xz), \ell(yz) \bigg) > 1/2
\]
By \eqref{k4def1}, \eqref{k4def3}, and \eqref{k4def4},
\[
\label{k4def6}
	\ell(xz) + \ell(wx) > 5/4
\]
Additionally, by \eqref{k4def1} and \eqref{k4def4},
\[
\label{k4def7}
	\ell(xy) + \ell(xz) > 3/2
\]
Let $G$ be the graph obtained from $K_4$ by subdividing the edge $xy$.
Let $v$ be the new vertex.
Let $vx$ and $vy$ be the new edges which are incident to $x$ and $y$, respectively.
By abuse of notation, let $\ell$ denote a compatible metric on $G$ which satisfies
\[
\label{k4def8}
	\ell(vx) + \ell(xz) = 10/9
\]
Note that \eqref{k4def1} and \eqref{k4def8} establish
\[
\label{k4def9}
	\ell(vy) + \ell(yz) > 25/18
\]

Let $A$ be the subgraph of $G$ consisting of the edges $xz$, $vx$, and $wx$.
Let $B$ be the subgraph of $G$ consisting of the edges $vy$, $yz$, and $wy$.
Let $\ell_A$ and $\ell_B$ be the induced metrics on $A$ and $B$, respectively.
By \eqref{k4def6}, \eqref{k4def8}, and Lemma \ref{lemclaw},
\[
	\lambda_1(A, \ell_A; v, w, z) \le \pi^2
\]
To complete the proof, it suffices to show that
\[
\label{k4def10}
	\lambda_1(B, \ell_B; v, w, z) \le \pi^2
\]
Suppose not.
Then by Lemma \ref{lemclaw} and \eqref{k4def9},
\[
\label{k4def11}
	\ell(yz) + \ell(wy) < 10/9
\]
Furthermore, by Lemma \ref{dirichletdelete} and \eqref{k4def5},
\[
\label{k4def12}
	\ell(vy) + \ell(wy) < 1
\]
Hence, by \eqref{k4def1}, \eqref{k4def4}, \eqref{k4def8}, and \eqref{k4def12},
\[
\label{k4def14}
	\ell(wz) > 7/18
\]
Also, by \eqref{k4def1}, \eqref{k4def8}, and \eqref{k4def12},
\[
\label{k4def15}
	\ell(yz) > 8/9
\]
By \eqref{k4def11} and \eqref{k4def15},
\[
\label{k4def16}
	\ell(wy) < 2/9
\]
Note that \eqref{k4def14} and \eqref{k4def16} imply that $\ell(wy) < \ell(wz)$.
By \eqref{k4def2}, this is a contradiction.
This establishes \eqref{k4def10}, completing the proof.
\end{proof}

\begin{Lemma}
\label{k4df}
Let $\ell$ be a metric on $K_4$ such that $L(K_4, \ell)=4$.
Assume $\ell(wx) + \ell(wz) \le 1$ and $\ell(wx) \ge 1/2$.
Then
\[
\label{k4df1}
	\lambda_1(K_4, \ell) \le \pi^2
\]
\end{Lemma}

\begin{proof}
Let $G$ be the tree obtained from $K_4$ by deleting $wz$ and by splitting $w$ and $z$ each into two leaves.
Let $w_x$, $w_y$, $z_x$, and $z_y$ denote the leaves of $G$ which are incident to $wx$, $wy$, $xz$, and $yz$, respectively.
Let $\ell_G$ be the induced metric.
Then
\[
	\lambda_1(K_4, \ell) \le \lambda_2(K_4, \ell; w, z) \le \lambda_2(G, \ell_G; w_x, w_y, z_x, z_y)
\]
Let $A$ be the tree obtained from $G$ by deleting $wx$.
Let $\ell_A$ be the induced metric.
Since $\ell(wx) \ge 1/2$, Lemma \ref{dirichletdelete} implies that
\[
	\lambda_2(G, \ell_G; w_x, w_y, z_x, z_y) \le \max \bigg( \lambda_2(A, \ell_A; w_y, z_x, z_y), \pi^2 \bigg)
\]
Since $\ell(wx)+\ell(wz) \le 1$, we have $L(A, \ell_A) \ge 3$.
Therefore, by Lemma \ref{lemdir},
\[
	\lambda_2(A, \ell_A; w_y, z_x, z_y) \le \frac{9 \pi^2}{L(A, \ell_A)^2} \le \pi^2
\]
This establishes \eqref{k4df1}.
\end{proof}

\begin{Lemma}
\label{k4abcd}
Let $\ell$ be a metric on $K_4$.
Assume $\ell(xy)+\ell(xz)+\ell(yz) \ge 2$ and $\ell(wx) \ge \ell(wy) \ge \ell(wz)$.
Also assume $\ell(wx)+\ell(wz) \ge 1$ and $\ell(wz) < 1/2$.
Then
\[
\label{k4abcd1}
	\lambda_1(K_4, \ell) \le \pi^2
\]
\end{Lemma}

\begin{proof}
By Lemma \ref{shorten}, we may assume $\ell(wy)=\ell(wz)$ and $\ell(wx)+\ell(wz)=1$.
Let $G$ be the graph obtained from $K_4$ by deleting $wx$, $wy$, and $wz$.
Let $\ell_G$ be the induced metric.
Then $G$ is a cycle graph and $L(G, \ell_G) \ge 2$.
Therefore,
\[
	\lambda_1(G, \ell_G) \le \pi^2
\]
Moreover, there is a corresponding eigenfunction $\phi$ in $H^1(G, \ell_G)$ such that $\phi(y)=\phi(z)$.
Note that $\ell(wx) > 1/2$, and subdivide $wx$.
Let $v$ denote the new vertex.
Let $vx$ and $vw$ denote the new edges which are incident to $x$ and $w$, respectively.
By abuse of notation, let $\ell$ be the compatible metric such that $\ell(vx)=1/2$.
Note that
\[
	\ell(vw) + \ell(wy) = \ell(vw) + \ell(wz) = 1/2
\]
Identify $vx$ with the interval $[0,1/2]$ so that $v$ is identified with zero and $x$ is identified with 1/2. 
Extend $\phi$ to $vx$ so that if $t$ is in $vx \sim [0,1/2]$, then
\[
	\phi(t) = \phi(x) \sin \pi t
\]
Identify $vw$ with the interval $[0, \ell(vw)]$ so that $v$ is identified with zero and $w$ is identified with $\ell(vw)$. 
Also identify $wy$ with the interval $[\ell(vw), 1/2]$ so that $w$ is identified with $\ell(vw)$ and $y$ is identified with $1/2$.
Extend $\phi$ to $vw \cup wy$ so that if $t$ is in $vw \cup wy \sim [0,\ell(vw)] \cup [\ell(vw), 1/2]$, then
\[
	\phi(t) = \phi(y) \sin \pi t
\]
Similarly, identify $wz$ with the interval $[\ell(vw), 1/2]$ so that $w$ is identified with $\ell(vw)$ and $z$ is identified with $1/2$.
Extend $\phi$ to $wz$ so that if $t$ is in $wz \sim [\ell(vw), 1/2]$, then
\[
	\phi(t) = \phi(y) \sin \pi t
\]
Let $\phi^+$ and $\phi^-$ be the positive and negative parts of $\phi$, respectively.
We may identify $\phi^+$ and $\phi^-$ with functions in $H^1(K_4, \ell)$.
Let $V$ be the subspace of $H^1(K_4, \ell)$ generated by $\phi^+$ and $\phi^-$.
Then
\[
	\lambda_1(K_4, \ell) \le \max_{f \in V} \frac{ \int_{K_4} |f'|^2 }{ \int_{K_4} |f|^2 } \le \pi^2
\]
This establishes \eqref{k4abcd1}.
\end{proof}

Now we can conclude the article by proving Theorem \ref{thmk4}.

\begin{proof}[Proof of Theorem 1.4]
We first prove that
\[
\label{k4upper}
	\Lambda_1(K_4) \le 16 \pi^2
\]
Let $\ell$ be a metric on $K_4$ such that $L(K_4, \ell)=4$.
By \eqref{scaledirichlet}, it suffices to prove that
\[
\label{thmk4pf1}
	\lambda_1(K_4, \ell) \le \pi^2
\]
There must be three edges in $K_4$ which form a cycle of length greater than or equal to two.
Without loss of generality, we may assume that
\[
	\ell(xy)+\ell(xz)+\ell(yz) \ge 2
\]
Suppose at least two edges of $wx$, $wy$, and $wz$ have length greater than or equal to 1/2.
Without loss of generality, we may assume that $\ell(wy) \ge 1/2$ and $\ell(wz) \ge 1/2$.
Then Lemma \ref{k4ef} establishes \eqref{thmk4pf1}.
Therefore, we may assume that at least two edges of $wx$, $wy$, and $wz$ have length less than 1/2.
Without loss of generality, we may assume that $\ell(wz) \le \ell(wy) < 1/2$.
If $\ell(wx) < 1/2$, then Lemma~\ref{k4def} establishes \eqref{thmk4pf1}.
Therefore, we may assume that $\ell(wx) \ge 1/2$.
If $\ell(wx) + \ell(wz) \le 1$, then Lemma \ref{k4df} yields \eqref{thmk4pf1}.
If $\ell(wx) + \ell(wz) \ge 1$, then applying Lemma \ref{k4abcd} completes the proof of \eqref{thmk4pf1}, establishing \eqref{k4upper}.

Let $A$ be a dipole graph with four edges.
By Example \ref{dipole},
\[
\label{k4lower}
	\Lambda_1(A) \ge 16 \pi^2
\]
Note $A$ can be obtained from $K_4$ by contracting two non-incident edges.
By a continuity result established by Band and Levy \cite[Appendix A]{BL}, this implies that $\Lambda_1(K_4) \ge \lambda_1(A)$.
Therefore \eqref{k4upper} and \eqref{k4lower} establish \eqref{thmk41}.
\end{proof}

\end{document}